\documentclass{article}
\usepackage{indentfirst}
\usepackage{amsthm}
\usepackage{amsfonts}
\usepackage{amssymb}
\usepackage{overpic}

\newcommand{\COLORON}{1}
\newcommand{\NOTESON}{0}
\newcommand{\Debug}{0}

\newcommand{\diam}{\mathrm{diam}}
\newcommand{\hhat}{\mathrm{hat}}
\newcommand{\cyl}{\mathrm{cyl}}
\newcommand{\ttop}{\mathrm{top}}

\usepackage[usenames]{color} 
\usepackage{amsthm,amssymb,amsmath,bbm,enumerate,graphicx,epsf,stmaryrd,accents}
\usepackage[bookmarks, colorlinks=false, breaklinks=true]{hyperref} 

\usepackage{authblk}

\newcommand{\comment}[1]{}
\newcommand{\COMMENT}[1]{}

\definecolor{darkgray}{rgb}{0.3,0.3,0.3}
\newcommand{\defi}[1]{{\color{darkgray}\emph{#1}}}

\newcommand{\acknowledgement}{\section*{Acknowledgement}}

\comment{
	\begin{lemma}\label{}	
\end{lemma}
\begin{proof}

\end{proof}

\begin{theorem}\label{}
\end{theorem} 
\begin{proof} 	

\end{proof}

}

\newtheorem{proposition}{Proposition}[section]
\newtheorem{definition}[proposition]{Definition}
\newtheorem{theorem}[proposition]{Theorem}
\newtheorem{corollary}[proposition]{Corollary}

\newtheorem{lemma}[proposition]{Lemma}
\newtheorem{observation}[proposition]{Observation}
\newtheorem{conjecture}{{Conjecture}}[section]

\newtheorem{problem}[conjecture]{{Problem}}

\newtheorem{examp}[proposition]{Example}

\newcommand{\FIG}{0}

\ifnum \NOTESON = 1 \newcommand{\note}[1]{ 

\hspace*{-30pt}
	{\color{blue}  NOTE: \color{Turquoise}{\small  \tt \begin{minipage}[c]{1.1\textwidth}  #1 \end{minipage} \ignorespacesafterend }} 
	
	}
\else \newcommand{\note}[1]{} \fi

\newcommand{\afsubm}[1]{ \ifnum \Debug = 1 {\mymargin{#1}}
\fi} 

\ifnum \Debug = 1 
\else  \fi

\ifnum \FIG = 1 \newcommand{\fig}[1]{Figure ``{#1}''}
\else \newcommand{\fig}[1]{Figure~\ref{#1}} \fi

\ifnum \FIG = 1 
\else  \fi

\ifnum \Debug = 1 \usepackage[notref,notcite]{showkeys}
\fi

\ifnum \COLORON = 0 \renewcommand{\color}[1]{}
\fi

\newcommand{\N}{\ensuremath{\mathbb N}}
\newcommand{\R}{\ensuremath{\mathbb R}}

\newcommand{\BS}{\ensuremath{\mathbb S}}

\newcommand{\cs}{\ensuremath{\mathcal S}}

\newcommand{\sm}{\backslash}

\newcommand{\cls}[1]{\ensuremath{\overline{#1}}}

\makeatletter
\DeclareRobustCommand{\cev}[1]{%
  \mathpalette\do@cev{#1}%
}
\newcommand{\do@cev}[2]{%
  \fix@cev{#1}{+}%
  \reflectbox{$\m@th#1\vec{\reflectbox{$\fix@cev{#1}{-}\m@th#1#2\fix@cev{#1}{+}$}}$}%
  \fix@cev{#1}{-}%
}
\newcommand{\fix@cev}[2]{%
  \ifx#1\displaystyle
    \mkern#23mu
  \else
    \ifx#1\textstyle
      \mkern#23mu
    \else
      \ifx#1\scriptstyle
        \mkern#22mu
      \else
        \mkern#22mu
      \fi
    \fi
  \fi
}

\makeatother

\newcommand{\nin}{\ensuremath{{n\in\N}}}

\newcommand{\pth}[2]{\ensuremath{#1}\text{--}\ensuremath{#2}~path}

\newcommand{\pths}[2]{\ensuremath{#1}\text{--}\ensuremath{#2}~paths}

\newcommand{\seq}[1]{\ensuremath{(#1_n)_{n\in\N}}} 

\newcommand{\g}{\ensuremath{G\ }}
\newcommand{\G}{\ensuremath{G}}

\newcommand{\Tr}[1]{Theorem~\ref{#1}}

\newcommand{\Sr}[1]{Section~\ref{#1}}

\newcommand{\Prb}[1]{Problem~\ref{#1}}
\newcommand{\Cr}[1]{Corollary~\ref{#1}}
\newcommand{\Cnr}[1]{Con\-jecture~\ref{#1}}

\newcommand{\Dr}[1]{De\-fi\-nition~\ref{#1}}

\newcommand{\fe}{for every}

\newcommand{\st}{such that}

\newcommand{\wrt}{with respect to}

\newcommand{\labtequ}[2]{
 \begin{equation} \label{#1} 	\begin{minipage}[c]{0.9\textwidth}  #2 \end{minipage} \ignorespacesafterend \end{equation} }

\newcommand{\mymargin}[1]{
 \ifnum \Debug = 1
  \marginpar{%
    \begin{minipage}{\marginparwidth}\small%
      \begin{flushleft}%
        {\color{blue}#1}%
      \end{flushleft}%
   \end{minipage}%
  }%
 \fi
}%

\newcommand{\extras}[1]{
 \ifnum \Debug = 1
\section{Extras} #1
 \fi
}%

\newcommand{\mySection}[2]{}

\begin{document}
	\title{Compact metric spaces with infinite cop number}
	
\author{Agelos Georgakopoulos\thanks{Supported by EPSRC grants EP/V048821/1 and EP/V009044/1.}}
\affil{  {Mathematics Institute}\\
 {University of Warwick}\\
  {CV4 7AL, UK}}

\date{\today}

\maketitle

\begin{abstract}
Mohar recently adapted the classical game of Cops and Robber from graphs to  metric spaces, thereby unifying previously studied pursuit-evasion games. He conjectured that finitely many cops can win on any compact geodesic metric space, and that their number can be upper-bounded in terms of the ranks of the homology groups when the space is a  simplicial pseudo-manifold. We disprove these conjectures by constructing a metric on $\mathbb{S}^3$ with infinite cop number. More problems are raised than settled.
\end{abstract}



\section{Introduction}

The game of Cops and Robber is one of the most studied games on graphs due its implications for the structure of the host graph, e.g.\ its tree-width \cite{SeyThoGra} or genus \cite{MohNot}. Well-known open problems include Meyniel's conjecture that $O(\sqrt{n})$ cops can catch the robber on any graph on $n$ vertices, and Schr\"oder's conjecture that $g + 3$ cops suffice on any graph of genus $g$ \cite{SchroCop}, while Mohar conjectures that $\sqrt{g}$ is the right order of magnitude \cite{MohNot}. See \cite{BonNowGam} for an extensive survey of the literature.

Mohar \cite{MohMin,MohGam} introduced a variant of the game taking place on an arbitrary compact geodesic metric space. This is similar to a game studied by Bollob\'as, Leader and Walters \cite{BoLeWaLio}. Other pursuit-evasion games like this taking place on continuous spaces have a rich literature, with motivation coming from control theory and several  practical applications; see \cite{Isaacs,MohGam,Lewin} and references therein.


\medskip
Mohar's aforementioned game is a common generalisation of many other pursuit-evasion games previously studied. It is played on an arbitrary compact geodesic metric space $X$, and involves a number of cops trying to capture, or approach arbitrarily close, a single robber. Cops and robber can travel the same distance $\tau(n)$ in each step $\nin$, decided by the robber under the sole restriction that $\sum \tau(n) = \infty$. The \defi{cop number $c(X)$} is the minimum number of cops that can win the game on $X$, i.e.\ be able to come to arbitrarily small distance to the robber. The precise definitions are given in \Sr{the game}. Mohar \cite{MohGam} made the following conjecture:

\begin{conjecture}[{\cite{MohGam}}] \label{conj fin}
Every game space $X$ has a finite cop number. 
\end{conjecture}

Our first result is a counterexample to this, obtained by glueing together a sequence of graphs with diverging cop numbers, after appropriately re-scaling their edge-lengths. Ir\v{s}i\v{c}, Mohar \& Wesolek \cite{IrMoWeCopII} had previously proved that the unit ball in $\ell^2(\N)$ has infinite \defi{strong} cop number, defined analogously but requiring a cop to `catch' the robber, i.e.\ coincide with his position.

I perceive this counterexample as good news for the concept of cop number, as it shows that finiteness of  $c(X)$ is a non-trivial property of a metric space $X$. It could be useful to relate $c(X)$ with other properties of a metric space, and we pose some problems in this direction in \Sr{FP}. It would be particularly interesting to find implications of the finiteness of $c(X)$ on the structure of $X$. The best example of such a result I am aware of says that if a single cop of speed strictly less than that of the robber can catch the latter on a finite graph \G, then \g is Gromov-hyperbolic \cite{CCPP}. It would be interesting to extend this result to an arbitrary compact geodesic metric space.

\medskip
Mohar obtained upper bounds on the cop numbers of compact surfaces, which naturally led to another conjecture that we will disprove:
\begin{conjecture}[{\cite{MohGam}}] \label{conj pseudo}
Suppose that $X$ is an $n$-dimensional simplicial pseudomanifold, whose $i$th homology group $H_i(X)$ has rank $r_i$ for $i=1, \ldots, n$. Then $c(X)=O(n \sqrt{r_1 + \ldots + r_n})$.
\end{conjecture}

A \defi{simplicial pseudomanifold} is a simplicial complex in which each simplex is contained in an $n$-simplex, and each $(n - 1)$-simplex is contained in at most two $n$-simplices.

We provide counterexamples $X=X_k, k\in \N$ to \Cnr{conj pseudo}, where $X$ is a simplicial 3-complex homeomorphic to a 3-manifold, in fact to $\BS^3$. Thus $r_i=0$ \fe\ $i$, yet we have $c(X_k)>k$. In our counterexamples to both aforementioned conjectures  the robber can afford to choose a constant \defi{agility} function $\tau$.

\medskip
These counterexamples contrast the following interesting theorem of Ir\v{s}i\v{c}, Mohar \& Wesolek \cite{IrMoWeCop}: If $M$ is a compact manifold (of any dimension) with constant curvature $-1$, then $c(M) = 2$. This suggests that given a compact topological space $X$, the metric we put on $X$ can affect $c(X)$ more than its topology. 

If $X$ is a simplicial complex, we can endow it with the following natural metric $d$, called the \defi{simplicial metric}: we start by giving each $k$-simplex $S$ of $X$ a metric $d_S$ that makes $S$ isometric with the standard $k$-simplex, and let $d$ be the length metric on $X$ induced by the $d_S$. The aforementioned counterexamples to \Cnr{conj pseudo} can in fact be metrized like this. In particular, they have a piecewise linear metric. If we drop this restriction of having a simplicial metric, then we can endow $\BS^n, n\geq 3$ with a metric that makes its cop number infinite (\Cr{cor inf}). 

\medskip
The following is modest alternatives to Conjecture~\ref{conj fin}: 

\begin{problem}\label{pro fin}
Does every compact metrizable topological space $X$ admit a metric $d$ \st\ $c((X,d))$ is finite?
\end{problem}

The following could be a first step towards \Prb{pro fin}:
\begin{conjecture} \label{conj new}
Suppose that $X$ is a simplicial complex homeomorphic to a compact manifold, endowed with its simplicial metric. Then $c(X)<\infty$ (and even $c_0(X)<\infty$).
\end{conjecture}

This might be easy to prove using the well-known idea of guardable sets \cite{AigFroGam}.

\medskip
Recall that in our counterexamples to \Cnr{conj pseudo} we can achieve arbitrarily large $c(\BS^3)$ with a PL metric, or infinite $c(\BS^3)$ with an unrestricted metric. This raises

\begin{problem}\label{pro man}
Does every  finite-dimensional, compact, Riemannian manifold $M$ have finite $c(M)$?
\end{problem}

The results of this paper suggest that it is not possible to bound $c(M)$ here in terms of the volume of $M$ even if we fix the topology and diameter of $M$. But it would be interesting to obtain upper bounds on $c(M)$ involving other parameters, e.g.\ the curvature, injectivity radius, etc.

\section{Definitions}

\subsection{Metric spaces} \label{MS}

Let  $X=(X,d)$ be a metric space. The \defi{ball}  of radius $r$ around $x\in X$ is the set $\defi{B(x,r)}:=\{ y\in X \mid d(x,y)< r\}$.

Given a topological path $p: [0,1] \to X$, we define its \defi{length $\ell(p)$} to be the  supremum of $\sum_{0\leq i < n} d(p(t_i),p(t_{i+1}))$ over all finite sequences $0 = t_0 < t_1 < \ldots t_n = 1$. 

We say that $d$ is a \defi{length metric} (and that $(X,d)$ is a length space), if $d(x,y)= \inf \ell(p)$ holds \fe\ $x,y\in X$, where the infimum ranges over all topological \pths{x}{y}.
When $X$ is compact, then it is not hard to see that this infimum is always realised by some \pth{x}{y}, which path we call an \defi{$x$--$y$~geodesic}. In this case, i.e.\ when every $x,y\in X$ are connected by a geodesic, we say that $X$ is a \defi{geodesic metric space}.

If $d$ is not a length metric, we can still use it to define one when $X$ is path-connected: define the \defi{(intrinsic) length metric $d'$ induced by $d$} by $d'(x,y):= \inf \ell(p)$, where the infimum ranges over all topological \pths{x}{y}.

\subsection{The game} \label{the game}

Let $X=(X,d)$ be a compact, geodesic metric space, called the \defi{game space}, and let $k \geq 1$ be an integer. The \defi{Game of Cops and Robber} on $X$ with $k$ cops is defined as follows. The first player, who controls the robber, selects the initial positions for the robber and for each of the $k$ cops. Formally, this is a pair $(r^0,c^0)\in X^{k+1}$, where $r^0 \in X$ is the robber's initial position and $c^0 = (c^0_1, \ldots, c^0_k)\in X$  are the initial positions of the cops. The same player selects the \defi{agility function}, which is a map $\tau: \N\to \R_+$ and will specify the distances that the players can travel in each step. The agility function must allow for the total length travelled to be infinite, i.e.\ $\sum_{n \geq 1} \tau(n) = \infty$.
After the initial position and the agility function are chosen, the game proceeds as a discrete
game in consecutive steps. Having made $n-1$ steps ($n\geq 1$), the players have their positions
$(r^{n-1}, c^{n-1}, . . . , c^{n-1}) \in X^{k+1}$. In the $n$th step, 
the robber moves to a point $r^n\in X$ at distance at most $\tau(n)$ from his current position, i.e.\ $d(r^{n-1},r^n) \leq \tau(n)$. The destination $r^n$ is revealed to the second  player who is manipulating the cops. Then
each cop $C_i, i \in [k]$ is moved to a position $c^n_i$, also at distance at most $\tau(n)$ from its current position,
i.e.\ $d(c^{n-1}_i,c^n_i) \leq \tau(n)$. The game stops if $c_i^n = r^n$ for some $i \in [k]$. In that case, the value of the
game is 0 and we say that the cops have \defi{caught the robber}. Otherwise the game proceeds. If it never stops, the \defi{value} of the game is
$$v := \inf_{n\geq 0} \min_{i\in[k]} d(r^n, c^n_i).$$
If the value is 0, we say that the cops won the game; otherwise the robber won. Note that the cops can win even if they never catch the robber.

Given a game space $X$, let \defi{$c(X)$} be the smallest integer $k$ such that $k$ cops win the game on $X$ for every strategy of the robber. If such a $k$ does not exist, then we set $c(X) = \infty$. We call $c(X)$ the \defi{cop number} of $X$. Similarly we define the \defi{strong cop number} $c_0(X)$ as the smallest cardinal $k$ such that $k$ cops can always catch the robber.

\section{The counterexamples}

We now present our example of a game space $X$ with infinite cop number, disproving \Cnr{conj fin}:\\

{\bf Counterexample~1:} Let \seq{G}\ be a sequence of finite graphs with  $c'(G_n) \to \infty$, as provided e.g.\ by Aigner and Fromme \cite{AigFroGam}, where $c'(G)$ denotes the graph-theoretic variant of the cop number of \G, i.e.\ where the players must move to an adjacent vertex or stay still in each step. We can think of $G_n$ as an 1-complex, endowed with its simplicial metric as defined in the introduction. Let $G'_n$ be the re-scaling of $G_n$ where each edge is given length $\frac1{n \cdot \diam(G_n)}$, and note that $\diam(G'_n) \to 0$. Form a metric space $X$ by picking a vertex $w_n$ in each $G'_n$ and identifying them into one vertex $w$. Notice that $X$ is compact, as the  $G'_n$ converge to $w$. We claim that $c(X) = \infty$. Indeed, given $k\in \N$, the robber can win the game against $k$ cops as follows. He picks $n$ \st\ $c(G_n)>k$, and sets his agility function to be the constant equalling the length $\ell$ of the edges of $G'_n$. He then plays according to his winning strategy for the discrete game on $G_n$ as follows. 

Whenever any cop $C_i$  leaves $G'_n$, the robber pretends that $C_i$ idles at $w$. (The robber never leaves $G'_n$ himself.) Whenever $C_i$ moves to an interior point $x$ of an edge, the robber pretends that $C_i$ moved to the vertex closest to $x$, unless $x$ is the midpoint of an edge. In the latter case, the robber picks a geodesic $\gamma$ from the previous pretended position of $C_i$ to $x$. If $\gamma$ passes through a vertex $z$, then the  robber pretends that $C_i$ moved to $z$. Otherwise, he pretends that $C_i$ stays at his previous pretended position. If we start the game with pretended positions coinciding with the actual ones, it is straightforward to check by induction on the number of steps that 
\labtequ{pret}{the pretended position $p_i^n$ of $C_i$ at step $n$ is within distance at most $\frac1{2}$ from the actual position,}
and therefore that 
$p_i^n$ is in the closed neighbourhood of $p_i^{n-1}$ in $G_n$. Thus moving from one pretended position to the next is a legal move in the discrete game.


Therefore, by playing the game on $G'_n$ according to his strategy for $G_n$ based on the pretended cop positions, the robber will never get caught, and in fact will avoid ever being approached to distance less than $\frac1{2}$ by \eqref{pret}. \qed

\bigskip
Our counterexample disproving \Cnr{conj pseudo} is more involved, and needs some preparation. Given a game space $X$, and a compact subspace $S\subset X$, we will show that it is possible to `cap $S$ off' by attaching to it a metric space $\hhat(S)$ that does not decrease $c(X)$ and such that $S$ deformation-retracts to a point through $\hhat(S)$. For example, when $S$ is homeomorphic to $\BS^1$, then $\hhat(S)$ is a topological disc with boundary $S$. More generally,  $\hhat(S)$ is homeomorphic to the cone over $S$, but it is important to endow it with the right metric. We define it without reference to our game, or the ambient space $X$, as follows. We will only make use of this definition for $S$ homeomorphic to $\BS^{1}$ or $\BS^{2}$ later on, so the reader will lose nothing by assuming that this is the case.

\begin{definition} \label{def hat}
\normalfont Given a compact metric space $(S,d_S)$, and $h\in \R_+$, we define the \defi{$h$-hat $\hhat(h,S)$ over $S$} as the following metric space. We form $\hhat(h,S)$ out of two parts, the \defi{cylinder} and the \defi{top} (\fig{figShadow} below might be helpful). The cylinder \defi{$\cyl(h,S)$} is the cartesian product $S \times [0,h]$, metrized by the corresponding $\ell_1$ metric, i.e.\ $d_1((s,t),(s',t')):= d_S(s,s')+ |t-t'|$ for any $s,s'\in S$ and  $t,t'\in [0,h]$. The top \defi{$\ttop(h,S)$} is obtained from a cylinder $S \times [0,h+\diam(S))$ after adding a `cone point' $z$, and endowing it with a metric $d_2$ such that $d_2((s,t),z)$ converges to 0 as $t\to h+\diam(S)$ and $d_2((s,0),(s',0))= d_S(s,s')$ \fe\ $s,s'\in S$. This is easy to achieve with a definition similar to $d_1$ after appropriate scaling.

We obtain $\hhat(h,S)$ by identifying the top layer $S\times \{h\}$ of $\cyl(h,S)$ with the bottom layer $S\times \{0\}$ of $\ttop(h,S)$. We endow $\hhat(h,S)$ with the length metric $d'$ \defi{induced} by $d_1$ and $d_2$, i.e.\ we let $d'(x,y)$ be the infimum of the lengths of all \pths{x}{y}\ $p$, where the length of a subpath of $p$ contained in $\cyl(h,S)$, respectively $\ttop(h,S)$, is measured \wrt\ to $d_1$ (resp.\ $d_2$).  Note that $d'(s,s')=d_S(s,s')$ \fe\ $s,s'\in S$, and $\hhat(h,S)$ is compact.
\end{definition}

The following is immediate from the construction of $\hhat(h,S)$:
\begin{observation} \label{obs Sn}
If $S$ is homeomorphic to the $n$-sphere $\BS^n,n\geq 1$, then $\hhat(h,S)$ is homeomorphic to the $(n+1)$-disc $\{x\in \R^{n+1} \mid d(x,(0,\ldots,0) ) \leq 1\}$.
\end{observation}

\begin{definition}\label{def Xh}
Given a game space $(X,d)$ and a compact subspace $S\subseteq X$, we let \defi{$X_{h,S}$} be the metric space obtained from $X \cup \hhat(h,S)$ by identifying $S\subseteq X$ with the bottom layer $S\times \{0\}$ of $\cyl(h,S)$. We endow $X_{h,S}$ with the length metric induced by $d$ and $d'$ as above. 
\end{definition} 

Note that $X_{h,S}$ is still compact, and that its metric is an extension of $d$. Using the compactness, it follows that $X_{h,S}$ a geodesic metric space, i.e.\ a game space. A similar, but much more restricted, construction was used by Mohar \cite[\S 8.1]{MohGam}.

\begin{theorem} \label{lem hat}
Let $X$ be a game space and let $S\subseteq X$ be compact. 
Then there is $h\in \R_+$ such that $c(X_{h,S})\geq c(X)$. 
\end{theorem}

Before proving this let us see how it can be used to disprove Mohar's \Cnr{conj pseudo}:\\ 

{\bf Counterexample~2:} Given $k\in \N$, we let $\cs_k$ be an appropriately metrized surface with $c(\cs_k)>k$, which have been constructed by Mohar \cite{MohGam}. Applying \Tr{lem hat} we will transform $\cs_k$ into a homeomorph $X$ of $\BS^3$ with $c(X)\geq c(\cs_k)$, which disproves \Cnr{conj pseudo}.

For this, embed  $\cs_k$ topologically into $\BS^3$ in a standard way, i.e.\ so that the image ---which we still denote by $\cs_k$--- separates $\BS^3$ into two components $C_1,C_2$, each homeomorphic to a solid surface. Let $T'$ be a triangulation of $\cs_k$, and let $T$ be a triangulation of $\BS^3$ extending $T'$ \st\ all 0-cells of $T$ lie on $\cs_k$. 

By repeatedly applying \Tr{lem hat} to the 1-cells, then the 2-cells, and then the 3-cells of $T\sm T'$, we will produce a homeomorph of $\BS^3$, while keeping the cop number greater than $k$ throughout. To make this precise, pick an 1-cell $e$ of $T\sm T'$, with end-vertices $x,y$ say, and apply \Tr{lem hat} with $S=\{x,y\}$ and a large enough $h$. Note that $S$ is homeomorphic to $\BS^0$, and therefore $\hhat(h,S)$ is homeomorphic to $e$ by Observation~\ref{obs Sn}.

By repeatedly applying \Tr{lem hat} to the remaining 1-cells  of $T \sm T'$, we obtain a homeomorph $\cs^1$ of the union of $\cs_k$ with the 1-skeleton of $T$, with cop number at least that of $\cs_k$. After this, we go through the 2-cells of $T \sm T'$, and proceed similarly: for each such 2-cell $f$, we  apply \Tr{lem hat} with $S$ being the boundary of $f$, which is homeomorphic to $\BS^1$, and therefore $\hhat(h,\partial f)$ is homeomorphic to a disc. This yields a homeomorph $\cs^2$ of the union of $\cs_k$ with the 2-skeleton of $T$. Finally, we proceed similarly with the 3-cells of $T$, to obtain a homeomorph of $\BS^3$ with cop number at least that of $\cs_k$. 

\medskip
\Cnr{conj pseudo} does not explicitly clarify whether we are allowed to put an arbitrary geodesic metric on $X$ (which seems to me to be the intended meaning), or whether we must use the metric induced by that of the standard simplex. Without any restriction on the metric, we can even achieve infinite cop number; see \Sr{sec common}. But our construction can be easily modified to yield counterexamples even with the aforementioned restriction. For this, we start by constructing our own $\cs_k$ as follows. Let $G_k$ be a graph with graph-theoretic cop number $c'(G_k)$ greater than $k$. Mohar \cite[Theorem 8]{MohGam} proved that we still have  $c(G_k)>k$ in our metric-space version of the game when we think of $G_k$ as a 1-complex with each 1-cell isometric to the real interval $[0,1]$ (and this holds if we force the constant agility function $\tau \equiv 1$). Let $f$ be an embedding of $G_k$ into an orientable surface $\cs'_k$ of minimal possible genus. A classical result of Youngs \cite{Youngs} says that each face $F$ of $f$ must be homeomorphic to an open disc. We may assume that the closure $\cls{F}$ of $F$ is homeomorphic to a closed disc, because otherwise we can embed a few edges inside $F$, with end-vertices in $\partial F$, to separate $\cls{F}$ into closed discs. By \Tr{lem hat}, if we  subdivide these new edges into long enough paths, then $c(G'_k)\geq c(G_k)$ holds for the resulting graph $G'_k$. We then apply \Tr{lem hat} to each face-boundary $S$ of the resulting embedding of $G'_k$. We thereby modify the construction of $\hhat(h,S)$, to make it isometric with a simplicial complex each simplex of which bears the standard metric. This is possible by cutting $\hhat(h,S)$ into small enough 2-simplices. Hereby we use the fact that the cylinder $S \times  [0,r]$ where $S$ is a circle of integer length $n$, and $r$ the height of an equilateral triangle of side-length 1, can be tiled into $2n$ equilateral triangles. Notice that we are free to choose $h$ to be a multiple of $r$ in any application of \Tr{lem hat}. 

After doing so for every face, we obtain a simplicial complex homeomorphic to $\cs'_k$, with the desired kind of metric, i.e.\ a simplicial metric. 

We then continue with our construction from above, again modifying the construction of $\hhat(h,S)$ to make it isometric with a simplicial complex; this is possible by Remark~3 below. This completes our counterexamples to \Cnr{conj pseudo}. \qed

\medskip
{\bf Remark 1:} We can continue applying \Tr{lem hat} in a similar fashion, to obtain a homeomorph $X_n$ of $\BS^n, n>3$ with arbitrarily large $c(X_n)$: note that $\BS^{n-1}$ separates $\BS^n$ into two $n$-discs, and so applying 
\Tr{lem hat} twice with $S= X_{n-1}$ we obtain a homeomorph $X_n$ of $\BS^n$ with $c(X_n)\geq c(X_{n-1})$.
\medskip

{\bf Remark 2:} 
In Counterexample~2 the robber can afford to choose a constant agility function $\tau$. This is because in each application of \Tr{lem hat} the robber just applies the same $\tau$ on $X_{h,S}$ that he was using for $X$, and he starts with $\tau \equiv 1$ on $G_k$.

\section{Proof of \Tr{lem hat}}
We now prove our main technical tool, \Tr{lem hat}.

\medskip
Let $k< c(X)$, and let $\tau$ be an agility function that the robber can use to win against $k$ cops on $X$. Mohar \cite{MohMin} proves that it is always in favour of the robber to use decreasing agility functions, and so we may assume that $\tau$ is decreasing, and in particular $M:= \max_{n\in \N} \tau(n)$ exists and is finite. We set $h:= M+\diam(X)$ and construct $X':=X_{h,S}$ as in \Dr{def Xh} (any larger $h$ would do as well). 

Let us first introduce some notation. Each point $x$ in $\hhat(h,S)$ except for the cone point $z$ comes with two coordinates $(s,h)$. We define its \defi{trace} $\pi(x):=s$ and \defi{height} $hei(x):=h$. We extend $\pi$ and $hei$ to $X$ by letting $\pi(s)=s$ and $hei(x)=0$ there, and we also let $hei(z):= 2h$ and $\pi(z)=s$ for an arbitrarily chosen $s\in X$. Easily, 
\labtequ{hei}{$d(x,\pi(x))= hei(x)$ \fe\ $x\in X \cup \cyl(h,S)$.}
Each topological path $p: [0,1] \to X'- z$ gives rise to a path $\pi \circ p$ in $X$. It is straightforward to check that 
\labtequ{ell}{$\ell(\pi \circ p) \leq \ell(p)$.}
We will modify the robber's winning strategy against $k$ cops on $X$ into such a strategy for $X'$; this implies that $k< c(X')$, proving our statement. 

For this, we will play a variant of the game, where in addition to the $k$ cops $C_i$ we have $k$ \defi{shadow cops} $C'_i$, the positions $s^n_i$ of which are chosen by a third player, the (robber's) \defi{accomplice}. The moves of the accomplice can be thought of as part of the robber's strategy. The shadow cops will help us prove that the (true) cops win nothing by placing themselves in $X'\sm X$. The robber and shadow cops will only move within $X$. 

This game evolves as follows. The robber initiates the game following his winning strategy against $k$ cops on $X$, with the agility function $\tau$ chosen above, placing each shadow cop $C'_i$ at the same position $s^0_i=c^0_i$ as $C_i$.  From then on, the robber disregards the true cops $C_i$, and pretends to be playing against the shadow cops $C'_i$, playing according to his aforementioned strategy.\footnote{There is a well-known cop strategy of Aigner and Fromm \cite{AigFroGam} for guarding a shortest path in a graph, and we have stolen, and adapted, this strategy on behalf of the robber.} (Thus the robber ignores $X'\sm X$.) 

The cops $C_i$ disregard the shadow cops, and play on $X'$ according to their favourite strategy for trying to win against the robber. 

After each cop move, from $c^{n-1}_i$ to $c^{n}_i$ say, the accomplice moves $C'_i$ as follows: 
\begin{enumerate}
\item \label{a i} if $c^{n}_i \in X \cup \cyl(h,S)$, then move within $X$ as close to $C_i$ as possible, i.e.\ let $s^{n}_i$ be a point in the closed ball $\cls{B(s^{n-1}_i ,\tau(n))}$  in $X$ minimising the distance to $\pi(c^{n}_i)$ (it might be that $s^{n}_i \in X\sm S$); 
\item \label{a ii} if $c^{n}_i \in \ttop(h,S)$, then sit, i.e.\ let $s^{n}_i=s^{n-1}_i$.
\end{enumerate}

This completes the description of the game play. We claim that 
\labtequ{s c}{if $c^{n}_i \in X$, then $s^{n}_i=c^{n}_i$}
holds \fe\ $i\in [k]$ and \nin. This claim already implies that the true cops cannot catch the robber, i.e.\ that $c_0(X')\geq k$, because the robber follows a strategy that guarantees he is not caught by the shadow cops. In order to show our stronger statement that $c(X')\geq k$, we strengthen claim \eqref{s c} as follows:
\labtequ{s c e}{$d(s^{n}_i,\pi(c^{n}_i))\leq hei(c^{n}_i)$ holds \fe\ $i\in [k]$ and \nin.}
To see that this implies $c(X')\geq k$, notice first that as the shadow cops are restricted to $X$, and the robber follows his winning strategy for $X$ against them, there is $\epsilon>0$ \st\ no shadow cop ever comes $\epsilon$-close to the robber, i.e.\ $d(s_i^n,r^n)> \epsilon$. 
If a true cop comes $\epsilon/3$-close to the robber, i.e.\ $d(c_i^n,r^n)\leq \epsilon/3$,  for some $\epsilon<h$, then $d(c_i^n,X)=hei(c_i^n)\leq \epsilon/3$. But then  \eqref{s c e} yields $d(s^{n}_i,\pi(c^{n}_i))\leq \epsilon/3$, which combined with $d(c^{n}_i, \pi(c^{n}_i)) = hei(c^{n}_i)$ from \eqref{hei}, and the triangle inequality, contradicts that $d(s_i^n,r^n)> \epsilon$ (\fig{figShadow}). This means that the robber wins against the true cops in $X'$.

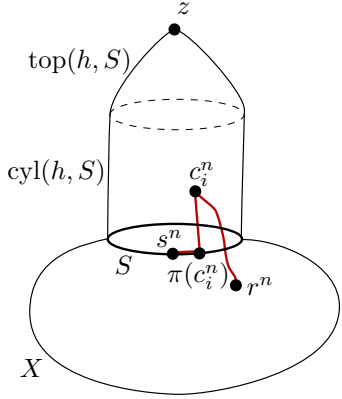
\begin{figure} 
\begin{center}
\begin{overpic}[width=.45\linewidth]{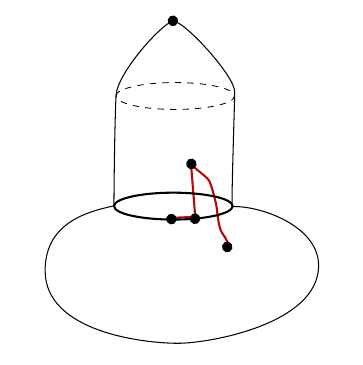} 
\put(65,30){$r^n$}
\put(45.7,33){$\pi(c^{n}_i)$}
\put(43,41.5){$s^n$}
\put(51,58){$c^{n}_i$}
\put(48,98){$z$}
\put(33,35){$S$}
\put(10,10){$X$}
\put(12,85){$\ttop(h,S)$}
\put(7,58){$\cyl(h,S)$}
\end{overpic}
\end{center}
\caption{Using $d(s_i^n,r^n)> \epsilon$ to prove $d(c_i^n,r^n)> \epsilon/3$.} \label{figShadow}
\end{figure}

\medskip
Thus it only remains to prove \eqref{s c e}, and we do so  for each $i\in [k]$ by induction on $n$. It is true for $n=0$ as $s^{0}_i= c^0_i= \pi(c^{0}_i)$ by the robber's choice of initial positions.

Assuming \eqref{s c e} holds for $n-1$, we prove it for $n$ by distinguishing the following cases.
\begin{itemize}
	\item[Case 1:] \label{O i}  $c^{n-1}_i \in X$. \\
	In this case our inductive hypothesis yields $s^{n-i}_i= c^{n-i}_i$. Let $p$ be a $c^{n-1}_i$--$c^{n}_i$~geodesic in $X'$. By \eqref{ell}, we have $\ell(\pi \circ p) \leq \ell(p)$, and therefore\\ $d(s^{n-1}_i, \pi(c^{n}_i)) \leq d(c^{n-1}_i,c^{n}_i)$ because $\pi(c^{n}_i)$ is the endpoint of $\pi \circ p$. Thus it would be an allowable move for $C'_i$ to move to $\pi(c^{n}_i)$, and therefore $C'_i$ chose $s^{n}= \pi(c^{n}_i)$, and so $d(s^{n}_i,\pi(c^{n}_i))=0$. Here we used the fact that $c^{n}_i\in X \cup \cyl(h,S)$ because $c^{n-1}_i \in X$ and $h>M\geq \tau(n)$.
	\item[Case 2:] \label{O ii}  $c^{n-1}_i \in \ttop(h,S)$ or $c^{n}_i \in \ttop(h,S)$.\\
	In the former case we have 
	$hei(c^{n-1}_i)\geq h \geq \tau(n)+ \diam(X)$. This implies $hei(c^{n}_i)\geq \diam(X)$, as a player cannot decrease his height by more than the distance they are  allowed to travel in a step. The latter inequality also holds trivially if $c^{n}_i \in \ttop(h,S)$. Thus \eqref{s c e} holds since $s^{n}_i,\pi(c^{n}_i)\in X$.
	\item[Case 3:] \label{O iii} $c^{n-1}_i \in \cyl(h,S)$ and $c^{n}_i \in X$.\\ 
	Let again $p$ be a $c^{n-1}_i$--$c^{n}_i$~geodesic in $X'$, and let $x$ be the first point of $p$ in $X$. Let $p_1, p_2$ be the two subpaths into which $x$ separates $p$. We claim that $d(s^{n-1}_i, c^n_i) \leq \ell(p) \leq \tau(n)$, and so $C'_i$ must choose $s^{n}_i = c^{n}_i$. 
	
To see this, note first that $p$ avoids $\ttop(h,S)$, because the latter is at distance more than $\tau(n)$ from the endpoint of $p$. This implies that 
	$$d(c^{n-1}_i,x)\geq d_1(c^{n-1}_i,x)= hei(c^{n-1}_i)+d(\pi(c^{n-1}_i), x)$$ by the definitions of our metrics. By the inductive hypothesis, we have $hei(c^{n-1}_i)\geq d(s^{n-1}_i,\pi(c^{n-1}_i))$, and combining this with the above inequality we obtain $d(c^{n-1}_i,x) \geq d(s^{n-1}_i,\pi(c^{n-1}_i))+d(\pi(c^{n-1}_i), x)$.
By the triangle inequality we have $d(s^{n-1}_i, c^n_i) \leq d(s^{n-1}_i, \pi(c^{n-1}_i)) + d(\pi(c^{n-1}_i),x) + d(x,  c^{n}_i)$, and combining with the previous inequality we obtain $d(s^{n-1}_i, c^n_i) \leq d(c^{n-1}_i,x) + d(x,  c^{n}_i)$. Since $p$ is a geodesic, the latter sum equals $\ell(p)$, proving our claim. Thus \eqref{s c e} holds in this case as well.	
	\item[Case 4:] \label{O iv} $c^{n-1}_i, c^{n}_i\in \cyl(h,S)$.\\
	In this case we have \mymargin{$\geq$}
$$d(c^{n-1}_i,c^{n}_i)\geq d_1(c^{n-1}_i,c^{n}_i)= | hei(c^{n-1}_i) - hei(c^{n}_i) | +d(\pi(c^{n-1}_i), \pi(c^{n}_i)).$$
In other words, the distance that $C_i$ traveled in step $n$ is at least the sum of the differences that his move has made to each of the two sides of inequality \eqref{s c e}. Thus $C'_i$ can choose a  $s^{n-1}_i$--$\pi(c^{n}_i)$~geodesic $g$, and move to a point $s^{n}_i$ on $g$ at distance $d(c^{n-1}_i,c^{n}_i)$ from $s^{n-1}_i$, to ensure that  \eqref{s c e} remains valid.\\  \qed
\end{itemize} 

{\bf Remark 3:} The conclusion of \Tr{lem hat} remains valid as is if we `inflate' our metric $d$ on $\hhat(h,S)$,  i.e.\ if we replace it with any metric $d^+$ on $X'$ with the following properties: (A) $d^+(x,y)\geq d(x,y)$ \fe\ $x,y\in X'$, (B)  $d^+(x,y)=d(x,y)$ \fe\ $x,y\in X$, and (C) 
 $d^+$ generates the same topology as $d$. Indeed, such a  change of metric restricts the moves of the true cops, while the robber and accomplice can ignore the change of metric since they are moving within $X$ only.

\section{A common counterexample} \label{sec common}

We now go one step further from Counterexamples~1 \& 2, and produce a game space homeomorphic to $\BS^3$ with infinite cop number, thus obtaining a strong simultaneous counterexample to Conjectures~\ref{conj fin} and \ref{conj pseudo}:
\begin{corollary} \label{cor inf}
There is a game space $X$ homeomorphic to $\BS^3$ with $c(X)=\infty$. 
\end{corollary}
\begin{proof}[Proof (sketch)]
We modify the construction in Counterexample 2, to produce game spaces $X_k$ homeomorphic to the closed 3-disc instead of $\BS^3$, by letting $T$ be a triangulation of a portion of $\BS^3$ instead of all of it. By rescaling the metric of $X_k$, we may assume that $\lim_{k\to \infty} \diam(X_k) = 0$. 

Embed all $X_k, k\in \N$ topologically into $\BS^3$, so that the images are disjoint, and each $X_k$ intersects the `equator' $\BS^1 \subset \BS^3$ along a subarc. Let us arrange these images so that they have a single accumulation point on $\BS^1$. Let $U$ be the union of $\bigcup_{k\in \N} X_k$ with the points of $\BS^1$ that are disjoint from the image of $\bigcup_{k\in \N} X_k$. It is easy to see that the intrinsic length metric $d$ on $U$ induced by the metrics of $\BS^1$ and the $X_k$ is a geodesic metric, and that $d$ coincides with  the original metric of $X_k$ when restricted to that subspace. Therefore, following the arguments of the proof of \Tr{lem hat}, with $S$ being a pair of points  $\partial X_k \cap \BS^1$, 
we obtain $c(U)=\infty$; indeed, the robber can choose to stay inside one of the $X_k$, chosen after the number of cops is fixed. Here, we are assuming the diameter of each $X_k$ to be much smaller than that of $\BS^1$.

It is easy to find two 2-discs $S_1,S_2$ in $\BS^3$, bounded by circles $C_1,C_2$, \st\ the complement of $S_1 \cup S_2 \cup U$ consists of two components $F_1,F_2$, each homeomorphic to a 3-disc. Indeed, we can obtain $C_1$ from $\BS^1$, by replacing the subarc inside the image of each  $X_k$ with an appropriate arc contained in the boundary of $X_k$ with the same end-points. Applying \Tr{lem hat} four times starting with $U$, and with $S$ being $C_1,C_2, \partial F_1$ and $\partial F_2$, in that order, we end up with the desired game space $X$. 
\end{proof}

\section{Further problems} \label{FP}

We conclude with some open problems that seek to make connections between the cop number of a space and some of its classical topological and metric properties.
\medskip

The doubling constant of a metric space $(X,d)$ is the minimal $k \in \N \cup \{\infty\}$ such that for all $x\in X$, and  $r > 0$, the ball $B(x,r)$ can be covered by at most $k$ balls of radius $r/2$. 

\begin{conjecture} \label{conj double}
Let $X$ be a  game space with finite doubling constant. Then $c(X) < \infty$.
\end{conjecture}

\begin{problem} \label{prb tdim}
Let $X$ be a compact topological space \st\ $c((X,d))<\infty$ holds for every geodesic metric $d$ compatible with the topology of $X$. Must $X$ have topological dimension at most 2?
\end{problem}

Recall that Ir\v{s}i\v{c}, Mohar \& Wesolek \cite{IrMoWeCopII} proved that the unit ball $B$ in $\ell^2(\N)$ has $c_0(B)=\infty$ but $c(B)=1$. This motivates

\begin{problem} \label{prb c0}
Suppose a game space $(X,d)$ satisfies $c_0(X)=\infty$ but $c(X)<\infty$. Must $X$ have infinite topological dimension?
\end{problem}

Call a metric space $(X,d)$ \defi{homogeneous}, if \fe\ $x,y\in X$ there is an isometry $i: X \to X$ with $i(x)=y$. 
\begin{conjecture} \label{conj hom}
Let $X$ be a homogeneous game space. Then $c(X) < \infty$.
\end{conjecture}

\acknowledgement{I thank George Kontogeorgiou for a discussion that triggered \Sr{sec common}.

\bibliographystyle{plain}
\bibliography{../collective}

\end{document}